\documentclass[12pt,a4paper,twoside]{article}

\usepackage[utf8]{inputenc}
\usepackage[T1]{fontenc}
\usepackage[english]{babel}

\usepackage[dvipsnames]{xcolor}

\usepackage[hcentering]{geometry}

\usepackage{enumitem}

\usepackage{amsmath,amsfonts,amssymb,amsthm}%

\usepackage
{graphicx}

\graphicspath{{Images/}}

\usepackage{fancyhdr}
\renewcommand{\phi}{\varphi}
\renewcommand{\theta}{\vartheta}
\renewcommand{\epsilon}{\varepsilon}

\newcommand{\field}[1]{\mathbb{#1}} 
\newcommand{\R}{\field{R}}
\newcommand{\Z}{\field{Z}}
\newcommand{\N}{\field{N}}

\newcommand{\A}{\field{A}}
\newcommand{\B}{\field{B}}
\newcommand{\Hb}{\field{H}}
\newcommand{\HM}{\field{H}_M}
\newcommand{\CC}{\mathcal{C}}
\newcommand{\BB}{\mathcal{B}}
\newcommand{\FF}{\mathcal{F}}
\newcommand{\PP}{\mathcal{P}}

\newcommand{\RR}{\mathcal{R}}
\newcommand{\HH}{\mathcal{H}}

\newcommand{\lmin}{\ell_\mathrm{min}}
\newcommand{\lmax}{\ell_\mathrm{max}}
\DeclareMathOperator{\Sp}{Sp}
\DeclareMathOperator{\dist}{dist}
\DeclareMathOperator{\sgn}{sgn}
\DeclareMathOperator{\meas}{meas}
\DeclareMathOperator{\arctanh}{arctanh}
\newcommand{\mydeg}{\mathfrak{D}}

\providecommand{\abs}[1]{\left\lvert#1\right\rvert}        
\providecommand{\norm}[1]{\left\lVert#1\right\rVert}    
\newcommand{\scal}[2]{\left\langle #1,#2\right\rangle}

\newcommand*{\diff}{\mathop{}\!\mathrm{d}}
\newcommand{\dd}{\diff}

\newtheorem{theorem}{Theorem}
\newtheorem*{theorem*}{Theorem}
\newtheorem{lemma}[theorem]{Lemma}
\newtheorem{prop}[theorem]{Proposition}
\newtheorem{corol}[theorem]{Corollary}
\theoremstyle{definition}
\newtheorem{defin}[theorem]{Definition}
\newtheorem{remark}[theorem]{Remark}

\title{Lower bound on the number of periodic solutions for asymptotically linear planar Hamiltonian systems}
\author{}
\date{}

\begin{document}
	\maketitle

\vspace{-1cm}
\centerline{\scshape Paolo Gidoni}
\medskip
{\footnotesize
	\centerline{Centro de Matemática, Aplicações Fundamentais e Investigação Operacional (CMAF – CIO),
		 }
	\centerline{Faculdade de Ciências da Universidade de Lisboa}
	\centerline{Campo Grande, Edificio C6,  	1749—016 Lisboa, Portugal}
} 
\bigskip

\centerline{\scshape Alessandro Margheri}
\medskip
{\footnotesize
	\centerline{Departamento de Matemática and Centro de Matemática, Aplicações Fundamentais }
	\centerline{e Investigação Operacional (CMAF – CIO), Faculdade de Ciências da Universidade de Lisboa}
\centerline{Campo Grande, Edificio C6,  	1749—016 Lisboa, Portugal}
}

\bigskip

\begin{abstract}
In this  work we prove the lower bound for the number of $T$-periodic solutions of an asymptotically linear planar Hamiltonian system. Precisely, we show that such a system, $T$-periodic in time, with $T$-Maslov indices $i_0,i_\infty$ at the origin and at infinity, has at least $\abs{i_\infty-i_0}$ periodic solutions, and an additional one if $i_0$ is even.
Our argument combines the Poincaré--Birkhoff Theorem with an application of topological degree.
We illustrate the sharpness of our result, and extend it to the case of second orders ODEs with linear-like behaviour at zero and infinity.
\end{abstract}

\bigskip
{\small
}
\newpage

\section{Introduction}

In the study of periodic solutions for planar Hamiltonian systems, we can traditional identify two main approaches: topological methods and variational ones. This alternative is very evident when we consider systems with twist between zero and infinity.

On the topological side, the main approach is based on the celebrated \emph{Poincaré--Birkhoff Theorem}, sometimes also called \emph{Poincaré's last geometric Theorem}. The Theorem assures the existence of two fixed points for every area-preserving homeomorphism of the planar annulus rotating the two boundary circles in opposite directions, and has led to several studies in Hamiltonian dynamics. We suggest \cite{DalReb} for an introduction to the result, and the references in  \cite{FSZ} for some recent applications.

On the variational side, a pivotal role is played by the seminal papers \cite{AZ,CZ}, employing \emph{Maslov's index}, also known as \emph{Conley--Zehnder index}, that have inspired a large number of generalizations (cf.~the book \cite{Abb}).
This approach applies to $2N$-dimensional Hamiltonian systems which are asymptotically linear at the origin and at infinity. To such systems it is possible to assign a couple of indices $i_0, i_\infty$, describing the behaviour of the two associated linear system.
Roughly speaking,  the index counts the half-rotations in $\Sp(1)$ made by the path $t\to \Psi(t)$ defined by the fundamental solution of a linear Hamiltonian system, in a given time interval $[0,T]$, cf.~Section \ref{sec:linear}. Twist corresponds to the case $i_0\neq i_\infty$: it assures the existence of a $T$-periodic solution, and that of a second one if the first is nondegenerate.

The connection between such variational results and the Poincaré--Birkhoff Theorem was already observed by Conley and Zehnder in \cite{CZ}, where they conjectured that the integer $\abs{i_\infty -i_0}$ is a measure of the lower bound of $T$-periodic solutions.

This relationship was made even clearer in \cite{MRZ}. Here, the rotational informations contained in the Maslov indices at zero and infinity were made explicit, and combined with  a modified Poincaré--Birkhoff Theorem, to obtain a higher multiplicity of periodic solutions, and partially proving the lower bound $\abs{i_\infty -i_0}$, under some conditions on the parity of the indices. Such results have later been applied and extended also to the case of resonance, cf.~\cite{GaMaRe,MRT}.

In this paper we continue this line of investigation, showing that the Maslov index contains even more information on the topological behaviour of the solutions of the systems. We use this to improve the result in \cite{MRZ}, finally obtaining the existence of at least $\abs{i_\infty -i_0}$ $T$-periodic solutions, independently by the parity of the indices. Moreover,   we show that asymptotic linearity, both at the origin and at infinity, can be replaced by a weaker condition of \lq\lq linear-like\rq\rq\ behaviour, requiring only that the vector field is bounded between two linear ones, having the same Maslov index.

\medbreak
In order to discuss more precisely our results, let us first introduce the framework of our work.
We consider a planar Hamiltonian system 
\begin{equation}\label{eq:HSplanar}
\dot z=JD_z  H(t,z)
\end{equation} 
where $D_z$ denotes the gradient with respect to the $z$-variable, and $J=\begin{pmatrix}
0& 1\\ -1 & 0
\end{pmatrix}$.
Concerning the regularity of the Hamiltonian function $H:\R^2\to \R,$ we make the following assumptions:
\begin{enumerate}[label={\textup{(H$_\mathrm{reg}$)}}]
\item The Hamiltonian function $H$ is continuous,  periodic in $t$ with period $T$, continuously differentiable in $z$ with $D_zH(t,\cdot)$ Lipschitz continuous uniformly in time.  \label{cond:h1}

	\end{enumerate}
As we anticipated, we are interested in the case when the system is asymptotically linear at zero and infinity, namely
\begin{enumerate}[label={\textup{(H$_0$)}}]
	\item For every $t\in[0,T]$, the Hamiltonian function $H(t,z)$ is twice differentiable at the origin $z=0$ with respect to the space variable $z$, with $D_z H(t,0)=0$ and $D_{zz} H(t,0)=\A(t)\in\CC([0,T],\R^{2\times 2})$ \label{cond:h0}
\end{enumerate}
\begin{enumerate}[label={\textup{(H$_\infty$)}}]
	\item There exists a matrix $\B(t)\in\CC([0,T],\R^{2\times 2})$ such that \begin{equation*}
	\displaystyle\lim_{\norm z \to +\infty}\frac{ \norm{D_zH(t,z)-\B(t)z}}{\norm{z}}=0
	\end{equation*}
 uniformly in $t$. 
	\label{cond:hinfty}
\end{enumerate}

Provided that the matrices $\A$ and $\B$ are $T$-nonresonant and that their $T$-Maslov indices $i_0\neq i_\infty$ are different, the variational results in \cite{CZ} assures the existence of at least one $T$-periodic solution of \eqref{eq:HSplanar}.

However, when the difference between the Maslov indices is higher, we can increase the multiplicity of $T$-periodic solutions quite straightforwardly, by an iterated application of the Poincaré--Birkhoff Theorem.
In this way, when the Maslov indices are both odd, we obtain the existence of at least $\abs{i_\infty -i_0}$ periodic solutions. Since the twist produced when $\abs{i_\infty -i_0}=1$ is sufficient to prove the existence of a periodic solution, we would intuitively expect a unitary increase of the twist to be sufficient to find an additional solution, hence giving the existence of $\abs{i_\infty -i_0}$ periodic solutions also in the general case.
However the classical Poincaré--Birkhoff Theorem is not the right tool to exploit this additional twist: with this approach we find only $\abs{i_\infty -i_0}-1$ solutions if only one of the indices is even, and $\abs{i_\infty -i_0}-2$ solutions if both indices are even. 

\medbreak

Searching for these additional solutions, in \cite{MRZ} it was proposed a variation of the Poincaré--Birkhoff  Theorem, where the rotation assumption on one of the boundaries was weakened, requiring the desired direction of rotation only in one point of the boundary.
As shown in the same paper, that result can be applied to planar Hamiltonian systems, with a weak rotation condition near the origin and a standard rotation condition at infinity. Such is the situation of asymptotically linear Hamiltonian systems with even Maslov index at the origin, for which the Theorem can be used to recover two additional solutions. 

Unfortunately, the same approach cannot be applied to the case in which we have a standard rotation condition at the origin and a weak rotation condition at infinity, as illustrated by the counterexample produced in \cite{CMMR_04}. This correspond to the case in which the Maslov index at infinity is even. 

\medbreak
The main purpose of this paper is to amend this situation and obtain an optimal lower bound, taking into account also the additional twist provided by an even index at infinity.  In addition to the Poincaré--Birkhoff  Theorem, we use an argument based on topological degree to show the existence of at least $\abs{i_\infty -i_0}$ solutions in all cases. Moreover, if $i_0$ is even, a characterization of the fixed point index for planar area-preserving maps \cite{Sim} allows to recover one additional solution.
More precisely, we prove the following (see Theorem~\ref{th:main}).
\begin{theorem*} 
	Let us consider the Hamiltonian system \eqref{eq:HSplanar}  and assume that \ref{cond:h1}, \ref{cond:h0}, \ref{cond:hinfty} are satisfied. Suppose that the linear systems at zero and infinity are $T$-nonresonant  and denote respectively with $i_0$ and  $i_\infty$ their $T$-Maslov indices. Then system \eqref{eq:HSplanar} has  at least $\abs{i_\infty -i_0}$ $T$-periodic solutions. Moreover, if $i_0$ is even, with $i_0\neq i_\infty$, then the number of solution is at least $\abs{i_\infty -i_0}+1$.
\end{theorem*}
This lower bound on the number of periodic solutions is sharp. In Remark \ref{rem:sharp} we illustrate how to construct suitable planar Hamiltonian systems with exactly the number of periodic solutions given by Theorem~\ref{th:main}.

\medbreak

Our method, based on topological degree,  gives a way  to properly characterize the weak twist generated by only partial rotation, as those produced by an even Maslov index. In this, as discussed in Remark \ref{rem:counter}, we clarify the counterexample in \cite{CMMR_04}, filling the missing part, which prevents the swap of the boundary condition in the modified Poincaré--Birkhoff Theorem in \cite{MRZ}. 

We remark that our method recovers the exact number of rotations for each periodic solution; also in this we improve the main result of \cite{MRZ}, where no rotation estimate was provided for the additional solution in case of even $i_0$.

To fully illustrate the topological nature of our result, in Section \ref{sec:linlike} we replace asymptotic linearity with a more general condition of linear-like behaviour at both zero and infinity. In such conditions, we require that the vector field is bounded by two linear ones, with the same Maslov index, in some neighbourhood of the origin, and/or of infinity. We enunciate these examples in the simpler case of second order ODEs, but, as discussed in Remark \ref{rem:linlike_planar}, the results hold also in the general case of a planar system \eqref{eq:HSplanar}.


\section{Rotation properties for linear and asymptotically linear systems}	\label{sec:linear}
\subsection{Rotation properties for linear systems}

 The  rotational properties of the linear $T$-periodic planar Hamiltonian system
\begin{equation}\label{eq:planar_linear}
\dot z=JL(t)z
\end{equation}
in terms of its Maslov index  are well known for the nonresonant case (see \cite{MRZ}).  In this section, we   review how they were obtained, with the aim to   present an    explicit formula   for  the the $T$-Poincaré map
$\PP_T(\phi,r)$ of the lift of system \eqref{eq:planar_linear} to $\R\times \R^+$, given by polar coordinates.

The periodicity of $\PP_T(\phi,r)$ in $\phi$ will allow us to introduce, for each fixed $r>0$, a degree characterization of $\FF(\phi,r)=\PP_T(\phi,r)-(\phi,r)$ for nonresonant systems. 

This computation   will lead quite directly   to our multiplicity results of the next section.  For completeness,  we also show  how adapt the  formula of $\PP_T$ to cover the resonant case.

\medbreak

We start by recalling  that the evolution of  system  \eqref{eq:planar_linear} is described by the fundamental solution  matrix $\Psi(t)$,  such that the solution of the Cauchy problem  \eqref{eq:planar_linear} with initial condition $z(0)=z_0$ is $\Psi(t)z_0$.

To analyse the rotational properties of the system, it is natural to  study its evolution in polar coordinates $(\phi,r)$; for our purposes we adopt \emph{clockwise} polar coordinates.  
This is done by   lifting    the linear vector field  \eqref{eq:planar_linear} from $ \R^2\setminus \{0\}$  to its universal covering  $(\Pi, \R\times\R^+ $), where the  covering projection $\Pi$ is given by $z=\Pi(\phi,r)=(r\cos\phi, -r\sin\phi)\neq 0$. We remark that since we define angles clockwise, the map $\Pi$ is orientation preserving. The dynamics of the  equivalent lifted system on the covering $\R\times \R^+$ is described 
by the Poincaré time map
\begin{equation} \label{eq:linear_Pmap}
\PP(t,\phi_0,r_0)=(\phi_0+\Theta(t,\phi_0),r_0 \RR(t,\phi_0))
\end{equation}
which satisfies $\PP(0,\cdot,\cdot)=I$, and for any $t$
\begin{equation*}
\Psi(t)z_0=\Pi(\PP(t,\phi_0,r_0))
\end{equation*}
where $z_0=(r_0 \cos \phi_0, -r_0 \sin \phi_0).$
The fact that the functions $\Theta$ and $\RR$ do not depend on the radial component is a consequence of the linearity of $\Psi(t)$.
The map $\PP$ is an homeomorphism of $\R\times\R^+$  such that $\Theta/2\pi$ measures the rotation of a solution around the origin in the interval $[0,t]$. 
Also, we notice that the map $\PP$ is an homeomorphism of $\R\times\R^+$  such that $\Theta/2\pi$ counts exactly the windings made by the solutions around the origin.

\medbreak
In what follows we are going  to translate in terms of $\Theta$ and $\RR$  some classical results on $\Psi$.
To do so, we first recall some notations and facts (see \cite{Abb,Gl,MRZ}).

The  fundamental matrix  $t\to \Psi(t)$ of \eqref{eq:planar_linear} describes for $t\in [0,T]$ a  continuous  path in  the symplectic group $\Sp(1)$ starting form  $\Psi(0)=I$.

We can represent $\Psi(t)$ in the covering space $(\R^+\times \R\times \R,\zeta)$ of $\Sp(1)$, with
\begin{equation*}
\zeta(\tau,\sigma,\theta)=P(\tau,\sigma)R(\theta)
\end{equation*}
Here 
\begin{equation*}
R(\theta)=\begin{pmatrix}
\cos \theta & \sin \theta \\
-\sin \theta & \cos \theta
\end{pmatrix}\in \Sp(1) 
\end{equation*}
is the clockwise rotation of angle $\theta$, and 
\begin{equation*}
P(\tau, \sigma)=
\begin{pmatrix}
\cosh \tau + \sinh \tau  \cos \sigma &  \sinh \tau \sin \sigma\\
\sinh \tau \sin \sigma & \cosh \tau - \sinh \tau  \cos \sigma
\end{pmatrix}\in \Sp(1)
\end{equation*}
is a positive definite symmetric matrix representing an hyperbolic rotation. Let us therefore describe the path $t\to\Psi(t)$ as
\begin{equation}\label{eq:psi_param}
\Psi(t)=P(\tau(t),\sigma(t))R(\theta(t))
\end{equation}
for suitable continuous  functions $\tau(t),\sigma(t),\theta(t)$ satisfying $\theta(0)=\tau(0)=0$.

This characterization of the path $\Psi(t)$ can be used to define the 
corresponding $T$-Maslov index. To do that, let us consider the reparametrization  $\tau=\tau(\rho)=\arctanh \sqrt{\rho}$. The variables $(\rho,\sigma,\theta)\in [0,1[\,\times [-\pi,\pi[\,\times[0,2\pi[\,$ provide a description in polar coordinates of the interior of a solid torus, and therefore the map $(\rho,\sigma,\theta)\mapsto \Phi(\tau(\rho),\sigma,\theta)$ defines a homeomorphism between $\{z\in \R^2, \abs{z}<1\}\times \mathbb{S}^1$ and  $\Sp(1)$. The set of resonant matrices
\begin{equation*}
\Gamma^0=\{M\in\Sp(1): \det (I-M)=0\}=\left\{(\rho,\sigma,\theta): \rho=\sin^2\theta \, \text{and}\, \abs{\theta}<\frac{\pi}{2}\right\}
\end{equation*}
corresponds to a surface that divides $\Sp(1)$ in two connected components, contractible in $\Sp(1)$, that we can identify as
\begin{align*}
\Gamma^-&=\{M\in\Sp(1): \det (I-M)<0\}=\left\{(\rho,\sigma,\theta): \rho>\sin^2\theta \, \text{and}\, \abs{\theta}<\frac{\pi}{2}\right\}\\
\Gamma^+&=\{M\in\Sp(1): \det (I-M)>0\}	\\
&=\left\{(\rho,\sigma,\theta): \rho<\sin^2\theta \, \text{and}\, \abs{\theta}<\frac{\pi}{2}\right\}\cup 
\left\{(\rho,\sigma,\theta):  \abs{\theta}\geq\frac{\pi}{2}\right\}
\end{align*}
Thus, each path $\Psi\colon [0,T]\to \Sp(1)$, with $\Psi(0)=I$ and $\Psi(T)\notin\Gamma^0$ can be prolonged, without crossing $\Gamma^0$, to a path $\hat \Psi$ arriving at $-I$, if $\Psi(T)\in \Gamma^+$, or in $\scriptstyle\begin{pmatrix} 2 & 0 \\ 0 & \scriptstyle\frac{1}{2} \end{pmatrix}$, if $\Psi(T)\in \Gamma^-$. The $T$-Maslov index $i_T(\Psi)$ of the path is the integer counting the counter-clockwise half-windings in $\Sp(1)$ made by the extension $\hat \Psi$. We remark that the choice of the extension does not change the index.
\medbreak

Let us now pick $\bar \theta \in [-\pi,\pi[\,$ such that $\bar \theta \equiv \theta (T) \pmod{2\pi}$, and let $\bar \tau=\tau(T)$, $\bar \sigma=\sigma(T)$. The matrix $P(\bar \tau,\bar \sigma)$ has eigenvalues $\{e^{-\bar \tau},e^{\bar \tau}\}$; we denote with $\{v_1,v_2\}$ the associated orthonormal basis such that $\theta_0:=\arg (v_1)\in [0,\pi[\,$.
Let us denote with $i_T(\Psi)\in \Z$ the $T$-Maslov index of \eqref{eq:planar_linear} and define $\ell=\lfloor i_T(\Psi)/2\rfloor$.  We have the following result.

\begin{lemma}\label{lem:poincaremap}
	The Poincaré map $\PP_T=\PP(T,\cdot,\cdot)$ associated to \eqref{eq:planar_linear} 
	 has the form $\PP_T(r,\phi)=(\phi+\Theta_T(\phi), r\RR_T(\phi))$, where
	\begin{align}
	\Theta_T(\phi)&:=\Theta(T,\phi)= \bar \theta -K\pi + g(\phi +\bar \theta)\\
	\RR_T(\phi)&:=\RR(T,\phi)= \sqrt{e^{-2\bar{\tau}}\cos^2(\phi+\bar{\theta}-\theta_0)
		+e^{2\bar{\tau}}\sin^2(\phi+\bar{\theta}-\theta_0)}, \label{eq:R_formula}
	\end{align}
 
\begin{equation}\label{eq:laps_count}
		K=\begin{cases} 
	2\ell &\text{if $i_T(\Psi)=2\ell$}\\
	2\ell &\text{if $i_T(\Psi)=2\ell+1$ and $\bar \theta< 0$}\\
	2\ell+2 &\text{if $i_T(\Psi)=2\ell+1$ and $\bar \theta> 0$}
	\end{cases}
\end{equation}
	and
	$g\colon \R\to\,\left]-\frac{\pi}{2},\frac{\pi}{2}\right[\,$ is the function satisfying:
	\begin{enumerate}[label=\textup{(P\arabic*)}]
		\item $g$ is odd with respect to any point $\alpha_0\in \theta_0 +\frac{\pi}{2}\Z$, and hence it is $\pi$-periodic; \label{prop:P1}
		\item in the half period $[\theta_0,\theta_0+\frac{\pi}{2}]$ we have $$g|_{[\theta_0,\theta_0+\frac{\pi}{2}]}(\alpha)=\arccos \frac{1+(e^{2\bar \tau }-1)\sin^2(\alpha-\theta_0)}{\sqrt{1+(e^{4\bar \tau }-1)\sin^2(\alpha-\theta_0)}}\in \left[0,\frac{\pi}{2}\right[\,$$\label{prop:P2}
\item if $i_T(\Psi)=2\ell$, then $\Psi(T)\in \Gamma^-$. Hence $\bar \theta \in \,]-\frac{\pi}{2},\frac{\pi}{2}[\,$ and
		\begin{align*}
		\max g=-\min g>\abs{\bar \theta} 
		&&	0<\max g +\abs{\bar \theta}<\pi
		\end{align*}\label{prop:P3}
		\item if $i_T(\Psi)=2\ell+1$, then  $\Psi(T)\in \Gamma^+$. Hence $\bar \theta \in [-\pi,\pi[\,\setminus \{0\}$ and
		\begin{align*}
		\max g=-\min g<\abs{\bar \theta} 
		&&	\max g +\abs{\bar \theta}>0
		\end{align*}\label{prop:P4}

	\end{enumerate}
	
\end{lemma}

\begin{proof}
The proof follows the same lines of Lemma 4 in  \cite{MRZ}, where, however, the radial component \eqref{eq:R_formula} was not studied. 

The path $t\to \Psi(t)$, parametrized as in \eqref{eq:psi_param}, is homotopically deformed in $\Sp(1)$   to a  path $t\to \tilde\Psi(t)$ with the same endpoints  for which  it is easier to compute  the rotations  of the solutions, which are the same as the ones of $t\to \Psi(t)$.
Such path    is defined as (see  \cite{MRZ})
	\begin{equation*}
	\tilde\Psi(t)=P(\tilde\tau(t),\tilde\sigma(t))R(\tilde\theta(t))
	\end{equation*}
	with
	\begin{align*}
	\tilde\tau(t)&=\tau(T)\max \left\{\frac{2t}{T}-1,0\right\} &
	\tilde\sigma(t)&=\sigma(t)\\
	\tilde\theta(t)&=\theta(T)\min \left\{\frac{2t}{T},1\right\} &&
	\end{align*}

The action of $\tilde{\Psi}$ on $\mathbb{\R}^2$  in $[0,\frac{T}{2}]$  is that of  a rigid rotation around $z=0$ of angle  $\theta(T)=\bar \theta -K\pi,$ where the integer $K$ is associated to the Maslov $T$-index of $\Psi$ as in \eqref{eq:laps_count}, whereas in $[\frac{T}{2},T]$    $\tilde{\Psi}$ acts as  
the hyperbolic rotation $P(\bar \tau,\bar \sigma)$.  This gives the structure of  the map $\PP_T$  described in the lemma. In fact, the angular term $\Theta_T$ is the sum  of the  constant $\bar \theta -K\pi$ corresponding  to the first half period, plus a bounded term $g(\phi+\bar \theta)$ produced by the hyperbolic rotation, and depending on the angle $\phi+\bar \theta$ after the rotation. The radial deformation $\RR_T$  is produced entirely by the hyperbolic rotation, and thus depends again on the angle $\phi+\bar \theta$.
	 
	 \medbreak
	 
The computation of the functions $g$ and $\RR_T$  proceeds as follows. 
	 
A point $\hat w\in \R^2$, represented as $(\hat r,\hat\phi)$ in polar coordinates, can be expressed in the coordinates $\{v_1,v_2\}$ as
\begin{equation*}
\hat w=\abs{\hat w}(y_1v_1+y_2v_2)
\end{equation*}
	 where
	 \begin{align*}
	 	y_1=\cos (\hat \phi-\theta_0) && y_2=\sin  (\hat \phi-\theta_0)
	 \end{align*}
Thus we can express the hyperbolic rotation $P(\bar \tau,\bar \sigma)$ as  	 
	 \begin{equation}
	 	P(\bar \tau,\bar \sigma)\hat w=\abs{\hat w}\left(y_1 e^{-\bar \tau}v_1+y_2 e^{\bar \tau}v_2\right)
	 \end{equation}
	 
Recalling that we are considering $\hat \phi= \phi+\bar \theta$, we obtain straightforwardly \eqref{eq:R_formula} for the radial component. 

To characterize the term $g$, we observe that it describes the rotation produced by the hyperbolic rotation $P(\bar \tau,\bar \sigma)$. Thus $g(\hat \phi)=0$ for every $\hat \phi\in \theta_0+\frac{\pi}{2}\Z$, corresponding to the points on the two invariant lines. Moreover, the vector field defined by $P(\bar \tau,\bar \sigma)$ is symmetric with respect to each of those axis, implying \ref{prop:P1}. Regarding the exact value of $g$, a straightforward computation leads to
\begin{equation}\label{eq:monotonetau}
\cos g(\hat \phi)=\frac{\scal{P(\bar \tau,\bar \sigma)\hat w}{\hat w}}{\abs{P(\bar \tau,\bar \sigma)\hat w}\abs{\hat w}}=\frac{y_1^2 e^{-\bar \tau}+y_2^2 e^{\bar \tau}}{\RR_T(\hat \phi)}
=\frac{1+(e^{2\bar \tau }-1)\sin^2(\hat \phi-\theta_0)}{\sqrt{1+(e^{4\bar \tau }-1)\sin^2(\hat \phi-\theta_0)}}
\end{equation}
We notice that $\cos g(\hat \phi)>0$. Moreover, our choice of the eigenvectors implies that $g$ is positive in a right neighbourhood of $\theta_0$. Combining these two facts we deduce \ref{prop:P2}; the values of $g$ on all the domain can be recovered combining \ref{prop:P1} and \ref{prop:P2}.
As to \ref{prop:P3}, \ref{prop:P4},  they follow  from the monotonicity in $\bar\tau$  of the range of $g,$ which is a consequence of \ref{prop:P2},  and from the fact that if $\Psi(T)$ belongs to the resonant surface $\Gamma_0$  then $\,\,\,\max g=-\min g=\abs{\bar \theta}$  (see \cite{MRZ} for the details). 
\end{proof}
\begin{remark}
A version of  Lemma \ref{lem:poincaremap}  can be stated  for  resonant systems. In this case    $1\in \sigma(\Psi(T)),$  and the path  $t\to \Psi(t)\in$ Sp(1), $\,\,t\in [0,T],\,\,\Psi(0)=I,\, $  ends on the resonant surface, that is   $\Psi(T)\in \Gamma^0.$ For such paths the   Maslov-type index (see \cite{Long})   is defined as the pair $(i_T(\Psi), \nu_T(\Psi))\in \Z\times \{1,2\},\,$ 
where $i_T(\Psi)$  is the minimum value of the  Maslov index attainable among  all the nonresonant  continuous paths  in   Sp(1) starting at $I$ for $t=0$   which are  sufficiently  $ C^0([0,T])$-close to  $\Psi,$  and  $\nu_T(\Psi)=\ker(I-\Psi(T))$  is the nullity of $\Psi(T).$ 
Of course, in such generalized setting  the Maslov index $i_T(\Psi)$ of a non resonant path  is identified with the pair $(i_T(\Psi),0).$ 

Using this definition, we see that  formulas \eqref{eq:psi_param} and \eqref{eq:R_formula} and properties \ref{prop:P1} and \ref{prop:P2} of  Lemma \ref{lem:poincaremap}  still hold in the resonant case   
with  \eqref{eq:laps_count}  modified as follows: 
\begin{equation}\label{eq:laps_count_res}
		K=\begin{cases} 
	2\ell &\text{if $(i_T(\Psi),\nu_T(\Psi)) \in (2\ell,1),$\,\, $ -\frac{\pi}{2}<\bar \theta < 0$  }\\
	2\ell+2  &\text{if $(i_T(\Psi),\nu_T(\Psi))=(2\ell+1,\{1,2\}),$\,\, $0\leq \bar \theta <\frac{\pi}{2}$}
	\end{cases}
\end{equation}   
In the second row   of \eqref{eq:laps_count_res},  $\nu_T(\Psi) = 2$ corresponds to the double resonance case $\Psi(T)=I,$ which occurs when $\bar\theta=0.$

Properties  \ref{prop:P3} and  \ref{prop:P4} must be adapted  as follows:

\begin{enumerate}[label=\textup{(P\arabic*')}, start=3]
\item  if $i_T(\Psi)=2\ell$, 
		\begin{align*}
		-\max g=\min g=\bar \theta\in\, \left]-\frac{\pi}{2},0\right[\,,
		&&	-\pi< g +\bar \theta\leq 0
\end{align*}\label{prop:P3'}		
\item if $i_T(\Psi)=2\ell+1$
		\begin{align*}
		\max g=-\min g=\bar \theta \in \left]0,\frac{\pi}{2} \right[\,,
		&&	0\leq  g +\bar \theta<\pi
		\end{align*}\label{prop:P4'}
\end{enumerate}
In particular,  we notice that  by \ref{prop:P4'} in  the case of double resonance $g\equiv 0$  and all solutions rotate counter-clockwise $\ell+1$ times around the origin. If  $\nu_T(\Psi)=1,$ by \ref{prop:P3'}   the counter-clockwise rotations of the solutions belong to the interval $[\ell, \ell+\frac{1}{2}[\,$, where $\ell$ corresponds to the rotation of the periodic solutions, and by \ref{prop:P4'} all the counter-clockwise rotations of the solutions belong to the interval $]\ell+\frac{1}{2}, \ell+1],$   where $\ell+1$ corresponds to the rotation of the periodic solutions.  
\end{remark}

\medskip

Let us consider the map $F\colon \R\times \,]0,+\infty[\,\to \R^2$ defined as
\begin{equation}\label{eq:PTM_field}
F(\phi,r)=(F_1(\phi,r),F_2(\phi,r))=(\Theta_T(\phi),r( \RR_T(\phi)-1))=\PP_T(\phi,r)-(\phi,r)
\end{equation}
We notice that $F(\phi,r)=(0,0)$ if and only if $\PP_T(\phi,r)=(\phi,r)$, meaning that $(\phi,r)$ corresponds to the initial point of a non-rotating $T$-periodic solution of \eqref{eq:planar_linear}.

Since $F$ is $2\pi$-periodic in the angular variable, we want to introduce a suitable notion of degree for every fixed radius.

\begin{defin} Let $f\colon \R\times \,]0,+\infty[\,\to \R^2$ be a continuous function. For every $r>0$ such that $f(\phi,r)\neq 0$ for every $\phi\in \R$, let us pick any continuous function $\tilde f^r\colon \R\times [0,+\infty[\,\to \R^2$, $2\pi$-periodic in the first variable, such that $\tilde f^r(\cdot,0)\equiv 0$ and $\tilde f^r(\cdot,r)\equiv f(\cdot, r)$.
	We define the degree $\mydeg(f,r)$ as
	\begin{equation*}
	\mydeg(f,r):=\deg(\tilde f^r\circ \Pi^{-1},\BB(0,r),0)	
	\end{equation*}
where $\deg$ denotes Brouwer's topological degree and $\tilde f^r\circ \Pi^{-1}\colon \R^2\to\R^2$ is properly defined due to the periodicity of $\tilde f^r$. We observe that $\mydeg(f,r)$ does not depend on the choice of $\tilde f^r$.
\end{defin}

The following two properties are direct consequences of the properties of the topological degree.

\begin{corol}\label{corol:annulus_degree}
Let $f\colon \R\times \,]0,+\infty[\,\to \R^2$ be a continuous function, $2\pi$-periodic in the first variable. Suppose that there exist $r_2>r_1>0$ such that, for every $\phi\in \R$, we have $f(\phi,r_1)\neq 0 \neq f(\phi,r_2)$. Then
\begin{equation*}
\deg(f\circ \Pi^{-1},\BB(0,r_2)\setminus \overline{\BB(0,r_1)},0)=\mydeg(f,r_2)-\mydeg(f,r_1)
\end{equation*}	
\end{corol}
\begin{corol} \label{corol:homotopy}
	 Let $\Hb \colon \R\times [0,1]\to \R^2\setminus\{0\}$ and $f_a,f_b\colon \R\times \,]0,+\infty[\,\to \R^2$  be three continuous functions, all of them $2\pi$-periodic in the first variable. If $\Hb(\phi,0)=f_a(\phi,r_a)$ and $\Hb(\phi,1)=f_b(\phi,r_b)$, then $\mydeg(f_a,r_a)=\mydeg(f_b,r_b)$.
	
\end{corol}

We remark that the map $F\circ \Pi^{-1}\colon \R^2\setminus \{0\}\to \R^2$ shall be considered just as an effective tool to study $F$, without looking for any special meaning as a flow on the annulus. 
Indeed, we are not conjugating $F$ with respect to $\Pi^{-1}$, but just composing it; hence our construction should not be confused with other properties, such as $\Psi(T)z=\Pi\circ \PP_T\circ \Pi^{-1}(z)$ for $z\neq 0$, which may be more familiar to the reader.

We now compute $\mydeg(F,r)$ in terms of the Maslov index associated to the corresponding linear system.

\begin{lemma} 
If $i_T(\Psi)=0$, then $\mydeg(F,r)=-2$ for every $r>0$. If $i_T(\Psi)\neq 0$, then $\mydeg(F,r)=0$ for every $r>0$.
\end{lemma}
\begin{proof}
By Lemma \ref{lem:poincaremap}, we have that  for 	$i_T(\Psi)\neq 0$, the map $\Theta_T\equiv F_1(\cdot,r)$ has constant sign, since $g$ is always smaller in modulus than $\bar \theta - K\pi$. Thus $\mydeg(F,r)=0$.

Let us  consider now the case $i_T(\Psi)= 0$. 
Since the system is linear, the signs of the two components $F_1(\phi,r)$ and $F_2(\phi,r)$ do not depend on $r$ and are repeated periodically in $\phi$ with period $\pi$.  As a consequence we will restrict ourselves to the interval  $\Lambda:=[\theta_0-\bar \theta,\theta_0-\bar \theta +\pi[\,.$ Let $h(\phi):=g(\phi+\bar\theta-\theta_0).$ By  \ref{prop:P2}   and \ref{prop:P3}  there exist $\phi_1, \phi_2\in \,]\theta_0-\bar \theta, \theta_0-\bar \theta+\frac{\pi}{2}[\,$  and $\phi_M\in \,]\phi_1,\phi_2[\,$ such that $h(\phi_1)=h(\phi_2)=\abs{\bar\theta}$ and $h(\phi_M)=\max g$. By \ref{prop:P3},  the symmetric points  of  $\phi_1, \phi_2,  \phi_M$  with respect  to the midpoint of $\Lambda,$ denoted respectively by $\phi_4, \phi_3, \phi_m,$ satisfy   $h(\phi_3)=h(\phi_4)=-\abs{\bar\theta}$ and $h(\phi_m)=\min g$,  with  $\phi_m\in \,]\phi_3, \phi_4[\,.$    A computation shows that $F_2(r, \phi_M)=F_2(r, \phi_m)=0$ and that $F_2(r,\phi)>0$ on $\Lambda$  iff $\phi\in \,]\phi_M,\phi_m[\,.$ 

As to  $F_1(r, \phi),$  by the properties of $g$ it follows that  its sign depends on the sign of $\bar\theta$.
More precisely, if $\bar\theta>0$  then 
 $F_1(r, \phi_i)=0, \,i=3,4, $ and $F_1(r,\phi) >0$ iff $\phi \in \Lambda\setminus [\phi_3, \phi_4],$  whereas, if $\bar\theta<0$ then $F_1(r, \phi_i)=0, \,i=1,2, $ and $F_1(r,\phi) >0$ iff $\phi \in \,]\phi_1, \phi_2[\,$.

We summarize the sign behaviour of the components of $F$ on $\Lambda$ in the following table, where we set $\hat\theta=\bar\theta-\theta_0$.\smallskip
\begin{equation*}
\begin{array}{lccccccc}
	& [\hat\theta,\phi_1[\,& [\phi_1,\phi_M[\, &[\phi_M,\phi_2]  & \,]\phi_2,\phi_3[\,  & [\phi_3,\phi_m[\,  & [\phi_m,\phi_4]  & \,]\phi_4,\hat\theta+\pi[\,   \\[6pt]
F_1(\cdot,r) &\sgn\bar\theta&\geq 0&\geq 0&\sgn\bar\theta&\leq 0&\leq 0&\sgn\bar\theta \\[4pt]	
F_2(\cdot,r)  &<0&< 0&\geq 0&>0&>0&\leq 0&<0\\		
\end{array}
\end{equation*}\smallskip
It follows that $\mydeg(F,r)=-2$.
\end{proof}

The same line of reasoning applies to the maps $F+(2M\pi,0)$, whose zeros correspond to $T$-periodic solutions making exactly $-M$ clockwise windings.
\begin{corol} \label{corol:deg_linear} 
	Let $M\in \Z$. If $i_T(\Psi)=2 M$, then $\mydeg\bigl( F+(2M\pi,0),r\bigr)=-2$. If $i_T(\Psi)\neq 2M$, then $\mydeg \bigl(F+(2M\pi,0),r\bigr)=0$.
\end{corol}
 We remind the reader that $F_1$ measures clockwise rotation, whereas Maslov's index is associated to counter-clockwise half-rotations.

\subsection{Rotational properties for asymptotically linear systems}

Analogously to the case of linear system, we can consider the lift of the flow of system \eqref{eq:HSplanar} to  polar coordinates, and the associated Poincaré map $\PP(t,\phi,r)$. Then, we define the map $\FF\colon \R^1\times \,]0,+\infty[\,\to \R^2$ as
\begin{equation} \label{eq:def_genlift}
\FF(\phi,r)=(\FF_1(\phi,r),\FF_2(\phi,r))=\PP(T,\phi,r)-(\phi,r)
\end{equation}

\begin{lemma} \label{lemma:rot_zero}
	Let us assume that the Hamiltonian system \eqref{eq:HSplanar} satisfies condition \ref{cond:h0}, and that the matrix $\A$ is $T$-nonresonant with associated Maslov index $i_0=i_T(\Psi_\A)$. Then there exists $r_0>0$ such that, for every $0<r<r_0$, we have
	\begin{align*}
	2\ell\pi < &-\FF_1(\cdot, r) <2(\ell +1)\pi &&\text{if $i_0=2\ell+1$}\\
	(2\ell-1)\pi < &-\FF_1(\cdot, r) <(2\ell +1)\pi &&\text{if $i_0=2\ell$}
	\end{align*}
	and, moreover, for every $M\in\Z$
	\begin{equation*}
\mydeg \bigl(\FF+(2M\pi,0),r\bigr)=\begin{cases}
		-2 &\text{if $i_0=2M$ }\\
		0 &\text{if $i_0\neq 2M$ }
		\end{cases}
	\end{equation*}
\end{lemma}

\begin{lemma} \label{lemma:rot_infty}
	Let us assume that the Hamiltonian system \eqref{eq:HSplanar} satisfies condition \ref{cond:hinfty}, and that the matrix $\B$ is $T$-nonresonant with associated Maslov index $i_\infty=i_T(\Psi_\B)$. Then there exists $r_\infty>0$ such that, for every $r>r_\infty$, we have
	\begin{align*}
	2\ell\pi < &-\FF_1(\cdot, r) <2(\ell +1)\pi &&\text{if $i_\infty=2\ell+1$}\\
	(2\ell-1)\pi < &-\FF_1(\cdot, r) <(2\ell +1)\pi &&\text{if $i_\infty=2\ell$}
	\end{align*}
	and, moreover, for every $M\in\Z$
	\begin{equation*}
\mydeg \bigl(\FF+(2M\pi,0),r\bigr)=\begin{cases}
-2 &\text{if $i_\infty=2M$ }\\
0 &\text{if $i_\infty\neq 2M$}
\end{cases}
\end{equation*}
\end{lemma}
\begin{proof}[Proofs of Lemmata \ref{lemma:rot_zero} and \ref{lemma:rot_infty}]

The proofs of these results are quite standard, cf.~for instance \cite[Lemmata 1 and 2]{MRZ}.  By uniform estimates on the nonlinearities at zero and infinity, we deduce, by \ref{cond:h0} and \ref{cond:hinfty}, that the following properties hold uniformly on $\phi$:
\begin{align*}
	\lim_{r \to 0} \FF_1(\phi,r)=F^\A_1(\phi,1) && \lim_{r \to 0} \frac{\FF_2(\phi,r)}{r}=F^\A_2(\phi,1) \\
	\lim_{r \to +\infty} \FF_1(\phi,r)=F^\B_1(\phi,1) && \lim_{r \to +\infty} \frac{\FF_2(\phi,r)}{r}=F^\B_2(\phi,1) 
\end{align*}
where $F^\A$ and $F^\B$ are, respectively,   the maps \eqref{eq:PTM_field} for the linearizations   at zero and infinity. The rotational properties of $\FF_1$ follow by the two limits above on the left and by Lemma \ref{lem:poincaremap}. 

To prove the degree property at zero, let us first notice that the limits above assure us the existence of $r_0$ such that $\FF(\phi,r)\neq 0$ for every $\phi\in \R$, $0<r<r_0$. Then, for every $0<r<r_0$, we consider the homotopy $\HM\colon \mathbb{S}^1\times [0,1]\to \R^2\setminus\{0\} $ defined as
\begin{equation*}
\HM(\phi,\lambda)=\begin{cases}
F^\A(\phi,1)+(2M\pi,0)&\text{for $\lambda=0$}\\
\left(\FF_1(\phi, \lambda r)+2M\pi,\frac{\FF_2(\phi, \lambda r)}{\lambda r} \right) &\text{for $0<\lambda\leq\frac{1}{2}$}\\
\left(\FF_1(\phi, \lambda r)+2M\pi,\frac{\FF_2(\phi, \lambda r)}{1+(1-\lambda)(r-2)} \right) &\text{for $\frac{1}{2}<\lambda\leq 1$}
\end{cases}	
\end{equation*}
that is continuous because of the limits above. The desired property is a consequence of Corollaries \ref{corol:homotopy} and \ref{corol:deg_linear}.

The proof of the degree property at infinity follows the same line, using a suitable rescaling of the radial coordinate.
\end{proof}


\section{Main result}

We observe that the change of coordinates associated to the map $\Pi$ is not symplectic. However we can adjust this situation by a simple rescaling of the radial coordinate (cf.~\cite{R}), namely by considering $\widehat\Pi \colon \R\times \R^+\to \R^2\setminus \{0\}$ defined as
\begin{equation*}
\widehat{\Pi}(\phi, \hat r)=\bigl(\sqrt{2\hat r}\cos \phi, -\sqrt{2\hat r}\sin \phi \bigr)	
\end{equation*}
Let us therefore define the maps $\widehat{\PP}_T$ and  $\widehat{\FF}$ analogously to $\PP_T$ and $\FF$, but with respect to the projection $\widehat{\Pi}$. We observe that $\widehat{\PP}_T$ is area preserving and satisfies
\begin{equation*}
\PP_T(\phi_a,r_a)=(\phi_b,r_b) \quad\text{if and only if} \quad \widehat{\PP}_T(\phi_a,r_a^2/2)=(\phi_b,r_b^2/2)
\end{equation*}
From this, we deduce straightforwardly the following properties
\begin{enumerate}[label=\textup{\roman*)}]
	\item $(\phi_0,r_0)$ is a fixed point for $\PP_T$ with fixed point index $j$ if and only if $(\phi_0,r_0^2/2)$ is a fixed point for $\widehat{\PP}_T$ with the same fixed point index $j$.
	\item For every $(\phi,r)\in \R\times\R^+$, we have $\FF_1(\phi,r)=\widehat{\FF}_1(\phi,r^2/2)$. \label{item:rot}
	\item For every $r>0$, $M\in \Z$, we have $$\mydeg \bigl(\FF+(2M\pi,0),r\bigr)=\mydeg \bigl(\widehat\FF+(2M\pi,0),r^2/2\bigr)$$ \label{item:degree}
\end{enumerate}

Taking into account \ref{item:rot}, we are now ready to state a main corollary of the classical Poincaré--Birkhoff Theorem for planar Hamiltonian systems (cf.~\cite{FU,R}).

\begin{theorem} \label{th_PBhamiltoniano}
	Let us consider the planar Hamiltonian system \eqref{eq:HSplanar} satisfying \ref{cond:h1}. We define the map $\FF$ as in \eqref{eq:def_genlift} and assume that there exist $r_1,r_2>0$ and $M\in \Z$ such that, for every $\phi\in[0,2\pi[\,$
\begin{equation}\label{eq:twistcond}
	\FF_1(\phi,r_1)<2\pi M<\FF_1(\phi,r_2)
\end{equation}
Then the system \eqref{eq:HSplanar} has at least two $T$-periodic solutions, making exactly $M$ clockwise rotations around the origin, and satisfying $r_1<\abs{z(0)}<r_2$ (or $r_2<\abs{z(0)}<r_1$).
\end{theorem}

We also recall the following result for the fixed point index of area-preserving maps (cf.~\cite[Prop~1]{Sim}).

\begin{prop}\label{prop:simon}
	Let $u$ be an isolated stationary point of an area preserving flow $f\colon U(\subseteq \R^2)\to \R^2$. Then the fixed point index of $f$ in $u$ is less than or equal to $+1$.
\end{prop}

By \ref{item:degree}, we have the following corollary.

\begin{corol} \label{corol:degree_lift}
	Let $(\phi_0,r_0)$  be an isolated zero of $\FF+(2M\pi,0)$. Then the fixed-point index of $\FF+(2M\pi,0)$ in  	$(\phi_0,r_0)$ is  less than or equal to $+1$.
\end{corol}

Combining these two results with the rotational properties discussed in the previous section, we have the following result.

\begin{theorem} \label{th:main}
	Let us consider the Hamiltonian system \eqref{eq:HSplanar}  and assume that \ref{cond:h1}, \ref{cond:h0}, \ref{cond:hinfty} are satisfied. Suppose that the linear systems at zero and infinity are $T$-nonresonant  and denote respectively with $i_0$ and  $i_\infty$ their $T$-Maslov indices. Then system \eqref{eq:HSplanar} has  at least $\abs{i_\infty -i_0}$ $T$-periodic solutions. Moreover, if $i_0$ is even, with $i_0\neq i_\infty$, then the number of solutions is at least $\abs{i_\infty -i_0}+1$.
\end{theorem}
\begin{proof}
	Let us assume, without loss of generality, that $i_0<i_\infty$, The case $i_0>i_\infty$ can be studied analogously, whereas the case $i_0=i_\infty$ is trivial.
	Let $\lmin, \lmax\in \Z$ be the two integers satisfying
	\begin{equation}
		2(\lmin-1)\leq i_0< 2\lmin \leq 2\lmax < i_\infty \leq 2(\lmax+1)
	\end{equation}
	We are going to prove the following
	\begin{enumerate}[label=\textup{(\alph*)}]
		\item For every $M\in \Z$ such that $\lmin \leq -M \leq \lmax$, system \eqref{eq:HSplanar} has at least two $T$-periodic solutions making exactly $M$ clockwise windings around the origin. \label{step:mainth1}
		\item If $i_\infty$ is even, namely $i_\infty=2(\lmax+1)$, then  system \eqref{eq:HSplanar} has at least one $T$-periodic solution making exactly $-(\lmax +1)$ clockwise windings around the origin. \label{step:mainth2}
		\item If $i_0$ is even, namely $i_0=2(\lmin-1)$, then  system \eqref{eq:HSplanar} has at least two $T$-periodic solutions making exactly $1-\lmin $ clockwise windings around the origin. \label{step:mainth3}
	\end{enumerate}

To prove \ref{step:mainth1}, let us notice that by Lemmata  \ref{lemma:rot_zero} and \ref{lemma:rot_infty} we obtain that 
condition \eqref{eq:twistcond} is satisfied for every $M$ such that $\lmin \leq M \leq \lmax$. Hence \ref{step:mainth1} follows directly from Theorem \ref{th_PBhamiltoniano}.

To prove \ref{step:mainth2} and \ref{step:mainth3}, we recall that  $T$-periodic solutions of \eqref{eq:HSplanar} making exactly $M$ clockwise rotations correspond to the zeros of the map $[\FF -(2M\pi,0)]\circ \Pi^{-1}\colon \R^2\setminus\{0\}\to \R^2$.
By Lemmata \ref{lemma:rot_zero} and \ref{lemma:rot_infty}, we deduce that
	\begin{equation}\label{eq:deg_estimate1}
\mydeg \bigl(\FF+(2(\lmax+1)\pi,0),r\bigr)=\begin{cases}
0 &\text{for $0<r<r_0$}\\
-2 &\text{for $r>r_\infty$ }
\end{cases}
\end{equation}
Applying Corollary \ref{corol:annulus_degree} to $f=\FF+(2(\lmax+1)\pi,0)$ we obtain
\begin{equation*}
\deg\bigl([\FF+(2(\lmax+1)\pi,0)]\circ \Pi^{-1},\BB(0,r_\infty+1)\setminus \overline{\BB(0,r_0/2)},0\bigr)=-2
\end{equation*}	
Then, by the properties of topological degree, there exists at least a zero for $[\FF+(2(\lmax+1)\pi,0)]\circ \Pi^{-1}$ in $\BB(0,r_\infty+1)\setminus \overline{\BB(0,r_0/2)}$, proving \ref{step:mainth2}.

To prove \ref{step:mainth3}, we proceed analogously, noticing that 
	\begin{equation} \label{eq:deg_estimate2}
	\mydeg \bigl(\FF+(2(\lmin-1)\pi,0),r\bigr) =\begin{cases}
-2 &\text{for $0<r<r_0$}\\
0 &\text{for $r>r_\infty$ }
\end{cases}
\end{equation}
and therefore
\begin{equation*}
\deg\bigl([\FF+(2(\lmin-1)\pi,0)]\circ \Pi^{-1},\BB(0,r_\infty+1)\setminus \overline{\BB(0,r_0/2)},0\bigr)=2
\end{equation*}	
Hence, by the property of topological degree we deduce that either $[\FF+(2(\lmin-1)\pi,0)]\circ \Pi^{-1}$ has an unique zero in the annulus $\BB(0,r_\infty+1)\setminus \overline{\BB(0,r_0/2)}$ with fixed-point index $2$, or it has at least two distinct zeros. Hence, to prove \ref{step:mainth3}, it suffices to show that a zero of $[\FF+(2(\lmin-1)\pi,0)]\circ \Pi^{-1}$ cannot have fixed-point index $2$.  Suppose, by contradiction, that there exists an isolated zero $\bar z\in \R^2\setminus \{0\}$ of $[\FF+(2(\lmin-1)\pi,0)]\circ \Pi^{-1}$ with fixed-point index $2$. We notice that  $\Pi$ is an orientation-preserving local diffeomorphism; therefore every point $(\bar\phi,\bar r)\in \Pi^{-1}(\bar z)$ is an isolated zero of $\FF+(2(\lmin-1)\pi,0)$ with the same fixed-point index $2$. However, by Corollary \ref{corol:degree_lift} this is not possible, since the zeros of $\FF+(2(\lmin-1)\pi,0)$ have fixed-point index at most equal to one.  Hence \ref{step:mainth3}.
\end{proof}
	
We notice that it is crucial to compute the degree of the map $\FF=\PP_T-I$ for the  lifted system on the halfplane, instead of the map $\Psi(T) -I$ for the planar system, since only in this way we can count the rotations of the recovered periodic solutions.

\begin{remark}[Sharpness of the lower bound] \label{rem:sharp}
We now show that the lower bound on the number of periodic solutions provided in Theorem \ref{th:main} is sharp.
We observe that examples for minimality in the general case can be recovered by suitably combining and adapting specific examples for the following three key situation. 
	\begin{enumerate}[label=\textup{\roman*)}]
		\item A system that is asymptotically linear in the origin with index $i_0$ even, and linear outside a given radius $r$ with $\abs{i_0-i_\infty}=1$, having  exactly two $T$-periodic solutions. This corresponds to point \ref{step:mainth3} of the proof. \label{it:min_zero}
		\item A system that is linear within a certain radius $r$ with index $i_0$ odd, and linear outside a given radius $R$ with $\abs{i_0-i_\infty}=2$, having  exactly two $T$-periodic solutions. This corresponds to point \ref{step:mainth1} of the proof. \label{it:min_pb}
		\item A system that is linear within a certain radius $R$ with index $i_0$ odd, and asymptotically linear at infinity with $\abs{i_0-i_\infty}=1$, having  exactly one $T$-periodic solution. This corresponds to point \ref{step:mainth2} of the proof. \label{it:min_infty}
	\end{enumerate}
To deal with these three cases, we consider autonomous systems. Since we have linear behaviour at infinity, for sufficiently small periods $T$ the only $T$-periodic orbits are fixed points. Such systems,  handling cases \ref{it:min_zero}, \ref{it:min_pb} and \ref{it:min_infty},
are illustrated respectively in Figures \ref{fig:zero}, \ref{fig:pb} and \ref{fig:infty}.
\end{remark}

\begin{figure}[tb]
	\centering
	\includegraphics[width=0.8\textwidth]{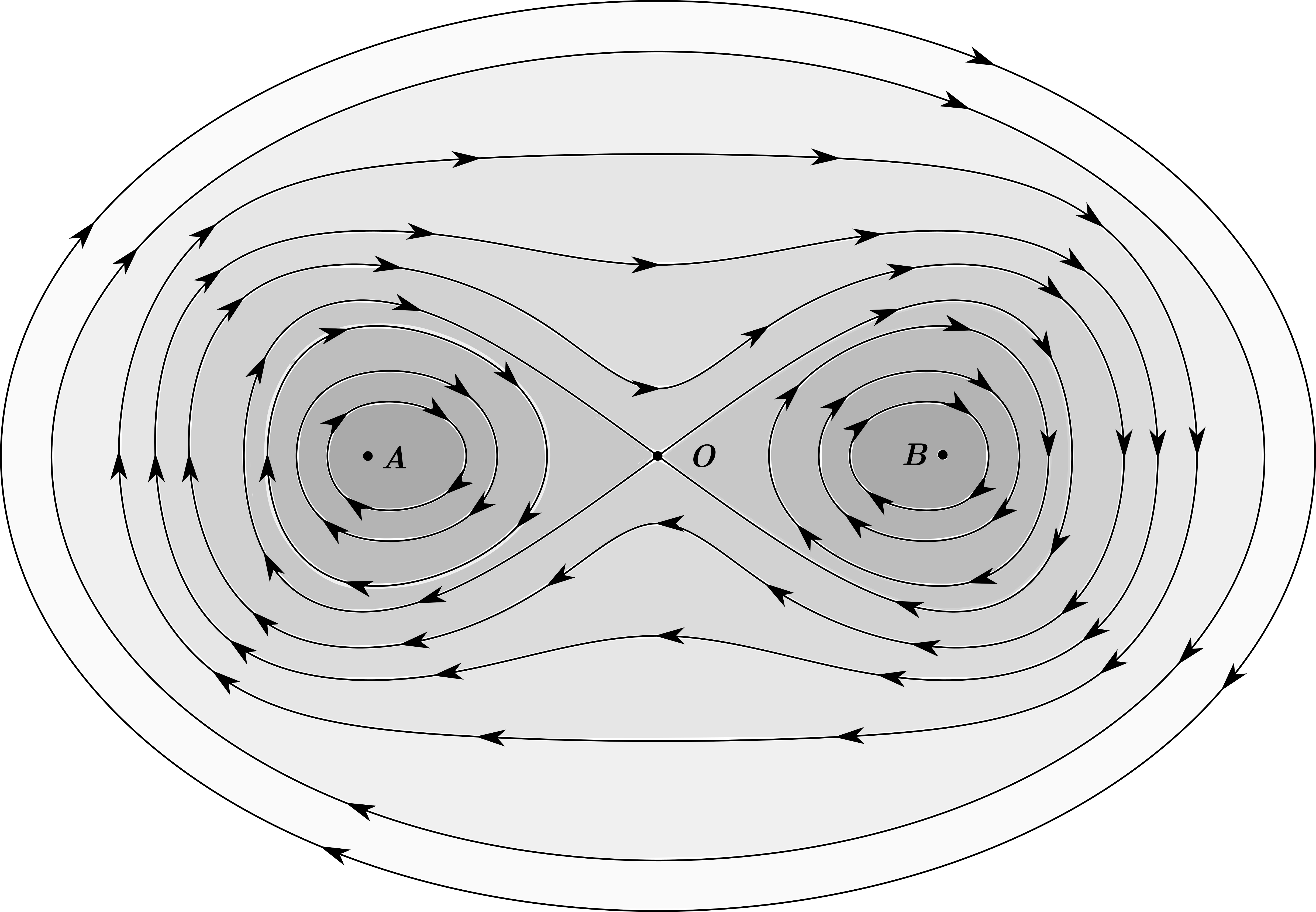}
	\caption{An Hamiltonian flow corresponding to case \ref{it:min_zero} of Remark \ref{rem:sharp}. For small periods $T$, the system has $i_0=0$, $i_\infty=-1$, and the only non-zero $T$-periodic orbits are the fixed points $A$ and $B$.}
	\label{fig:zero}
\end{figure}
\begin{figure}[tb]
	\centering
	\includegraphics[width=0.6\textwidth]{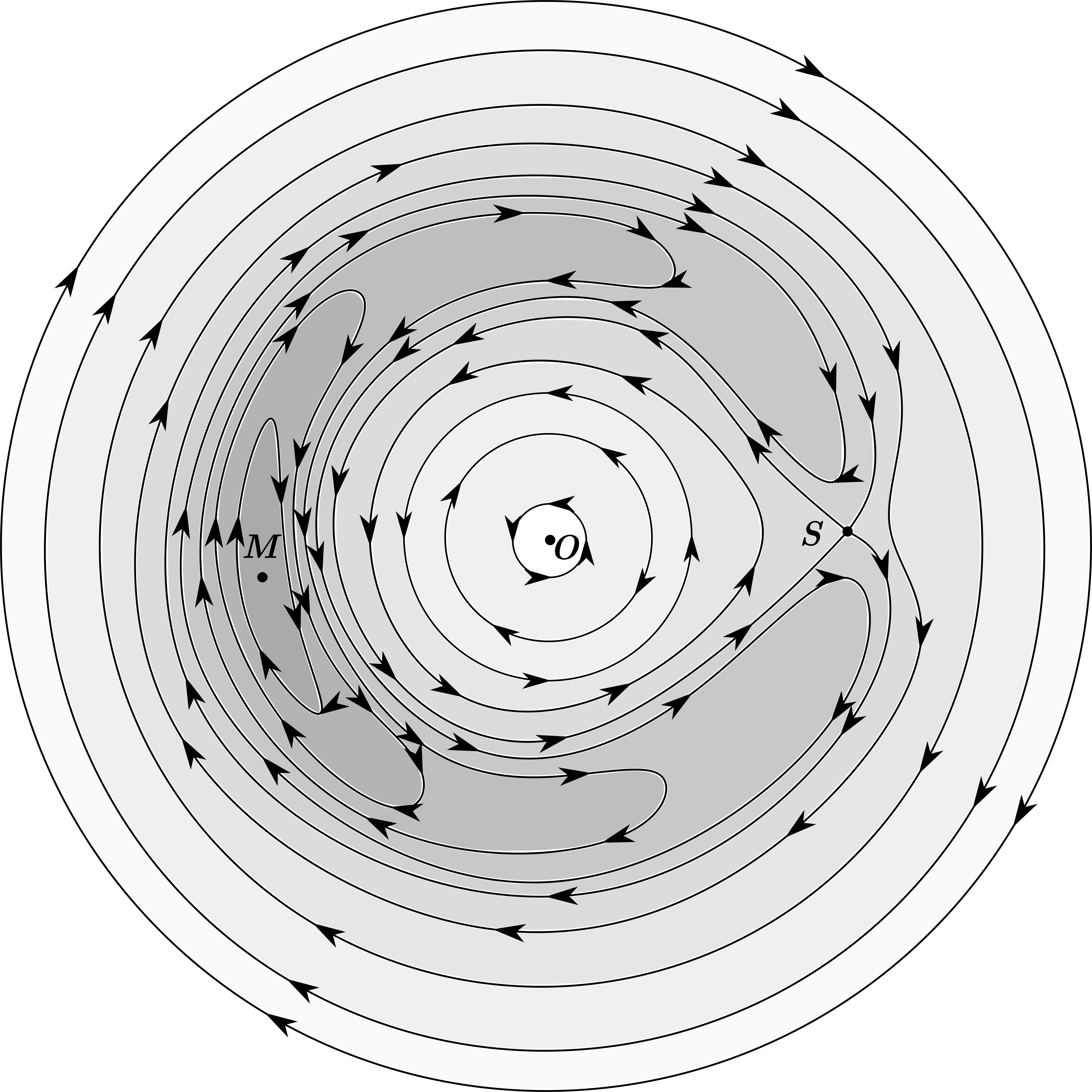}
	\caption{An Hamiltonian flow corresponding to case \ref{it:min_pb} of Remark \ref{rem:sharp}. For small periods $T$, the system has $i_0=1$, $i_\infty=-1$, and the only non-zero $T$-periodic orbits are the fixed points $M$ and $S$.}
	\label{fig:pb}
\end{figure}
\begin{figure}[tb]
	\centering
	\includegraphics[width=0.5\textwidth]{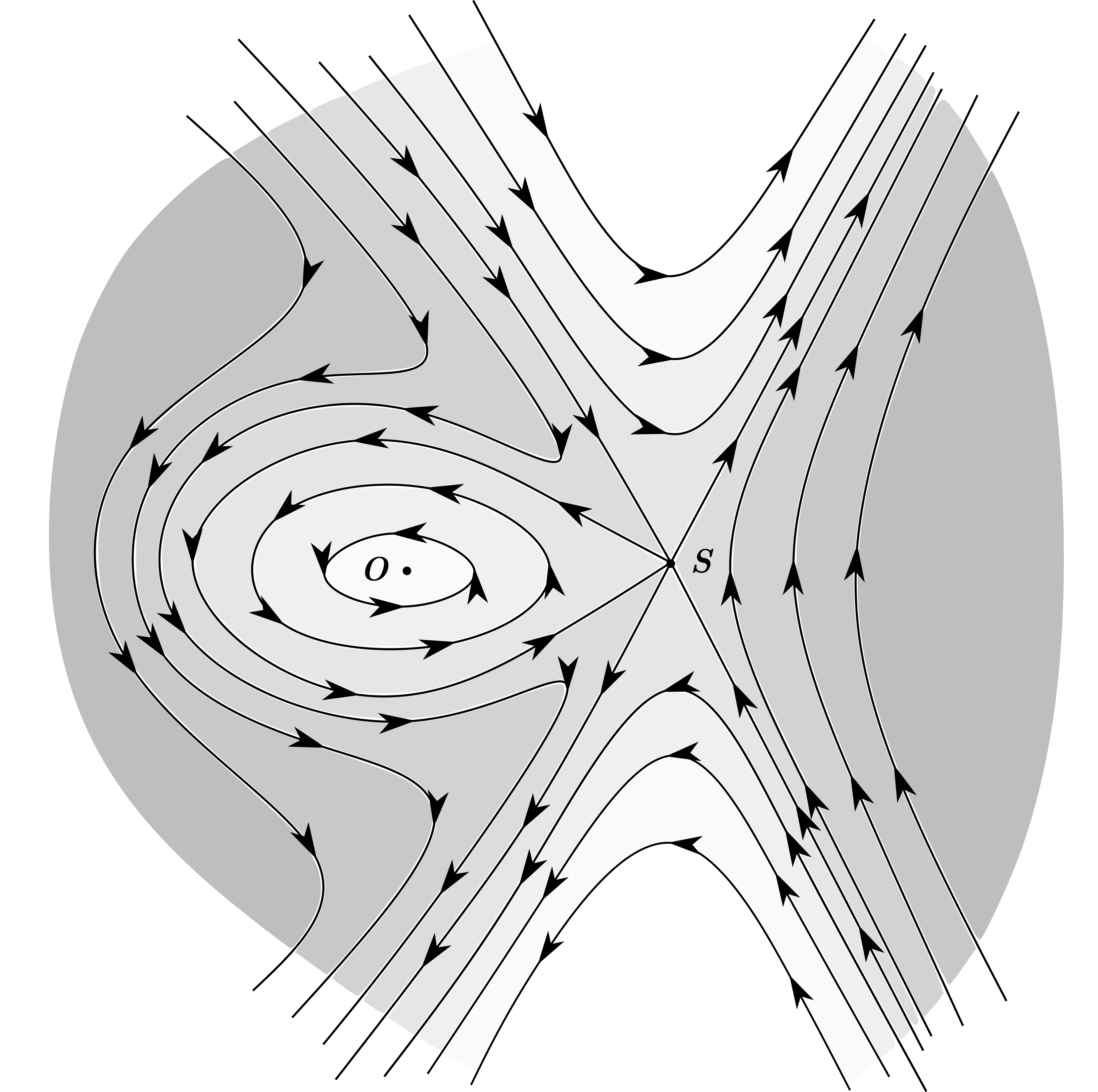}
	\caption{An Hamiltonian flow corresponding to case \ref{it:min_infty} of Remark \ref{rem:sharp}. For small periods $T$, the system has $i_0=1$, $i_\infty=0$, and the only non-zero $T$-periodic orbit is the fixed point $S$.}
	\label{fig:infty}
\end{figure}

\section{Second order ODEs and linear-like behaviour}
\label{sec:linlike}

A classical field of application of the Poincaré--Birkhoff Theorem, and related results, is provided by the second order differential equation
\begin{equation} \label{eq:sec_order}
\ddot x +q(t,x)x=0
\end{equation}
We make the following assumptions
\begin{enumerate}[label={\textup{(Q$_\mathrm{reg}$)}}]
	\item The function $q\colon \R\times \R\to \R$ is continuous, continuously differentiable in $x$, and periodic in $t$ with period $T$.	 \label{cond:q1}
\end{enumerate}
It is well know that equation \eqref{eq:sec_order} can be reformulated as a planar Hamiltonian system \eqref{eq:HSplanar}. Indeed, for  $z=(x,y)$, it suffices to set
\begin{equation}\label{eq:HSsec_order}
\begin{cases}
\dot x= -y\\
\dot y= q(t,x)x
\end{cases}
\end{equation}
so that the associated Hamiltonian function is 
\begin{equation*}
H(t,x,y)=\frac{1}{2}y^2 + \int_{0}^{x}q(t,\xi)\xi \dd \xi
\end{equation*}

The behaviour at zero and infinity is therefore controlled by the function $q$, and the conditions for asymptotic linearity at zero and infinity can be expressed as follows.

\begin{enumerate}[label={\textup{(Q$_0$)}}]
	\item There exists a continuous, $T$-periodic function $a(t)$ such that $\displaystyle\lim_{x\to 0} q(t,x)= a(t)$ uniformly in $t$. \label{cond:q0}
\end{enumerate}
\begin{enumerate}[label={\textup{(Q$_\infty$)}}]
	\item There exists a continuous, $T$-periodic function $b(t)$ such that $\displaystyle\lim_{\abs{x}\to +\infty} q(t,x)=b(t) $ uniformly in $t$. \label{cond:qinfty}
\end{enumerate}
In this framework, it is straightforward to obtain the following result.
\begin{corol} \label{cor:soODE}
	If the second order equation \eqref{eq:sec_order} satisfies conditions \ref{cond:q1}, \ref{cond:q0}, \ref{cond:qinfty}, then the associated Hamiltonian planar system \eqref{eq:HSsec_order} satisfies conditions \ref{cond:h1},  \ref{cond:h0}, \ref{cond:hinfty}, with the asymptotic behaviour at zero and infinity described respectively by the matrices 
	\begin{align} \label{eq:def_matrix}
	\A(t)=\begin{pmatrix}
	a(t) & 0\\ 0 & 1
	\end{pmatrix}&&
	\B(t)=\begin{pmatrix}
	b(t) & 0\\ 0 & 1
	\end{pmatrix}
	\end{align}
	In particular, if the planar linear Hamiltonian systems with matrices $\A(t)$ and $\B(t)$ are both $T$-nonresonant, with Maslov indices respectively $i_0$ and $i_\infty$, then the same conclusions of Theorem \ref{th:main} hold for \eqref{eq:sec_order}.
\end{corol}

Since our approach in based on topological methods, the condition of (asymptotic) linearity is used only as a tool to quickly identify a qualitative behaviour of the system, instead of as a strict technical requirement. To illustrate this situation we now introduce weaker conditions at zero and infinity for nonlinear systems with a linear-like behaviour.

\subsection*{Linear-like behaviour at the origin}

Let us first notice that, to study meaningful nonlinearities at the origin, we need first to weaken our regularity assumptions on $q$, excluding the origin from its domain.
\begin{enumerate}[label={\textup{(Q$_\mathrm{reg}^*$)}}]
	\item The function $q\colon \R\times (\R\setminus\{0\})\to \R$ is continuous, continuously differentiable in $x$, and periodic in $t$ with period $T$.	 \label{cond:q1_star}
\end{enumerate}

Before to introduce the condition of linear-like behaviour at the origin, let us first recall the monotonicity properties of Maslov index.
For any symmetric matrix $\A$, we write $\A>0$ if it is positive definite, and $\A\geq 0$ if it is positive semidefinite. Hence $\A\leq \B$ means that $\B-\A\geq 0$.
Moreover, for every time-dependent symmetric matrix $\A\in\CC( [0,T],\R^{2\times 2})$, let us denote with $i_\A$ the Maslov index associated to the linear system \eqref{eq:planar_linear} with matrix $\A$.

\begin{prop}[Monotonicity of Maslov index, cf.~\cite{Liu}] Let us consider two time-dependent symmetric matrices $\A,\B\in\CC([0,T],\R^{2\times 2})$. If $\A\leq \B$, then $i_{\A} +\nu_\A \leq i_{\B}.$	
	\label{prop:Maslov_monot}
\end{prop}

\begin{corol} \label{cor:conf_matrix}
	Let us consider three time-dependent symmetric matrix $\B,\B_1,\B_2\in\CC( [0,T],\R^{2\times 2})$. Assume that the matrices $\B_1,\B_2$ are $T$-nonresonant and have equal associated Maslov indices $i_{\B_1}= i_{\B_2}$. If $\B_1(t)\leq \B(t) \leq \B_2(t)$ for every $t\in [0,T]$, then also $\B$ is $T$-nonresonant and have  Maslov index $i_{\B}=i_{\B_1}= i_{\B_2}$.
\end{corol}
\begin{proof}
	We observe that, by Proposition \ref{prop:Maslov_monot} and since $\nu_{\B_1}=\nu_{\B_2}=0$, we have, 
\begin{equation*}
	i_{\B_1} \leq i_{\B} \leq i_{\B}+\nu_{\B} \leq i_{\B_2} =i_{\B_1} 
\end{equation*}
	Hence $\nu_{\B}=0$ and $i_{\B}=i_{\B_1}$.
\end{proof}

Corollary \ref{cor:conf_matrix} shows us that we can characterize the behaviour of a linear system, when it is controlled by two other linear systems with the same index. Our plan is to pursue this idea, by considering a nonlinear system bounded by linear ones. Hence, we replace condition \ref{cond:q0} with

\begin{enumerate}[label={\textup{(Q$_0^*$)}}]
	\item There exist $\bar r_0>0$ and two continuous, $T$-periodic functions $a_1(t),a_2(t)$ such that $a_1(t)< q(t,x) < a_2(t)$ for every $0<\abs{x}<\bar r_0$, $t\in[0,T]$. Moreover, the matrices 
	\begin{equation}
	\A_1(t):=\begin{pmatrix}
	a_1(t) & 0\\ 0 & 1
	\end{pmatrix}\qquad 
	\A_2(t):=\begin{pmatrix}
	a_2(t) & 0\\ 0 & 1
	\end{pmatrix}
	\end{equation}
	are $T$-nonresonant with Maslov index $i_0=i_{\A_1}=i_{\A_2}$. \label{cond:q0_star}
\end{enumerate}

We remark that \ref{cond:q1_star} and \ref{cond:q0_star} guarantee that the vector field associated to the planar system \eqref{eq:HSsec_order}  can be extended  to a locally Lipschitz vector field on  the whole plane by setting it equal to zero in the origin.

\begin{theorem} \label{th:linlike}
	Let us consider the Hamiltonian system \eqref{eq:HSsec_order}  and assume that \ref{cond:q1_star}, \ref{cond:q0_star} and \ref{cond:qinfty} are satisfied, with $i_\B=i_\infty$ for the matrix $\B(t)$ in \eqref{eq:def_matrix}. Then the same conclusions of Theorem \ref{th:main} hold.
\end{theorem}

\begin{proof}
The proof of the Theorem follows the same lines of that of Theorem \ref{th:main}. It suffices to prove that the twist property \eqref{eq:twistcond} and the degree properties \eqref{eq:deg_estimate1} and \eqref{eq:deg_estimate2} are satisfied. This is trivially true for their parts concerning $r>r_\infty$, since by Corollary \ref{cor:soODE} we know that system \eqref{eq:HSsec_order} is asymptotically linear at infinity, and so the argument used in the proof of Theorem \ref{th:main} holds.  	
We now prove that \eqref{eq:twistcond}, \eqref{eq:deg_estimate1} and \eqref{eq:deg_estimate2} hold also near the origin.
	
Let us observe that, for $\abs z<\bar r_0$, the planar vector field associated to \eqref{eq:HSsec_order} admits a linear bound. By Gronwall's Lemma, we deduce the existence of $r_0>0$ such that every solution $\tilde z(t)$ of \eqref{eq:HSsec_order}, with initial datum $\abs{\tilde{z}(0)}<r_0$, satisfies $\abs{\tilde{z}(t)}<\bar r_0$ for every $t\in [0,T]$.
	
We claim that the desired properties are satisfied for this choice of $r_0$. Indeed, let $\tilde z(t)=(\tilde x(t),\tilde y(t))$ be a solution of \eqref{eq:HSsec_order} with initial datum $\abs{\tilde{z}(0)}<r_0$. We set
\begin{align} \label{eq:point_linearized}
	\tilde{q}(t):=q(t,\tilde x(t))  &&	\tilde \A(t)=\begin{pmatrix}
	\tilde q(t) & 0\\ 0 & 1
	\end{pmatrix}
\end{align}

Clearly $\tilde z(t)$ is also a solution of the planar linear system
\begin{equation}
	\dot{\tilde z}=J\tilde \A(t)\tilde z
\end{equation}
On the other hand, since $\abs{\tilde{z}(t)}<\bar r_0$ for every $t\in [0,T]$, by \ref{cond:q0_star} we deduce that, for every $t\in [0,T]$ we have $a_1(t)< \tilde q(t) < a_2(t)$. Thus $\A_1(t)\leq \tilde\A(t) \leq \A_2(t)$ for every $t\in [0,T]$ and, by Corollary \ref{cor:conf_matrix}, we deduce that $\tilde\A(t)$ is $T$-nonresonant with Maslov index $i_0$. 
By Lemma  \ref{lemma:rot_zero}  we deduce that \eqref{eq:twistcond} holds for the point $\tilde z(0)$. By the generality of the choice of $\tilde z(0)$, \eqref{eq:twistcond} is true. 
	
	\medskip
To prove \eqref{eq:deg_estimate1} and \eqref{eq:deg_estimate2} a point-wise argument is no longer sufficient. Let us therefore fix $r^*\in (0,r_0)$ and consider the solutions of \eqref{eq:HSsec_order} with initial datum $\tilde{z}(0)$ expressed in polar coordinates by $(\alpha,r^*)$, with $\alpha\in[0,2\pi)$.

For every initial datum $\tilde{z}(0)$, we define the associated matrix $\tilde\A_\alpha(t)$ as in \eqref{eq:point_linearized}. Proceeding as in Section \ref{sec:linear}, we define the fundamental matrix $\Psi_\alpha(t)$ associated to the system $\dot{\tilde z}=J\tilde \A_\alpha(t)\tilde z$. We observe that every solution $\tilde z$ of \eqref{eq:HSsec_order} satisfies  $\tilde z(T)=\Psi_\alpha(T)\tilde z(0)$, where $\alpha$ varies accordingly to the initial datum.

Since $\A_1(t)$ and all the matrices $\tilde \A_\alpha$ have the same Maslov index, we deduce $\Psi_{\A_1}(T)$ and all the $\Psi_\alpha(T)$ are in the same component $\Gamma^\bullet\in\{\Gamma^+,\Gamma^-\}$ of $\Sp(1)$.
Since each of these components in contractible, there exists a homotopy $\HH\colon \mathbb{S}^1 \times [0,1]\to \Gamma^\bullet$ such that, for every $\alpha\in\mathbb{S}^1$
\begin{align*}
\HH (\alpha,0)=\Psi_{\A_1}(T) && \HH (\alpha,1)=\Psi_{\alpha}(T)	
\end{align*}

Let us denote with $\PP_{\A_1}(t,\phi,r)$ the Poincaré time map associated, as in \eqref{eq:linear_Pmap}, to the linear Hamiltonian system with matrix $\A_1(t)$, and analogously we denote with $\PP_{\alpha}(t,\phi,r)$ the ones corresponding to the matrices $\tilde \A_\alpha(t)$.
We define the map $j\colon  \R^2\setminus\{0\}\to \R\times\R^+ $ as the only map such that $\Pi\circ j=I$ and 
\begin{equation}\label{eq:def_j}
j(\Psi_{\A_1}(T)z)=\PP_{\A_1}(T,\phi,r)	
\end{equation}
where $z=(r \cos \phi, -r \sin \phi)$. We observe that, in general,  $\Pi\circ j=I$ defines $j$ up to horizontal translation of multiples of $2\pi$, corresponding to the  windings around the origin. Equation \eqref{eq:def_j} gives us the exact number of windings to consider, and so determines  $j$ univocally.
Moreover, since the number of windings is determined by the Maslov index, we deduce that, for the same choice of $j$, we also have, for every $\alpha\in\mathbb{S}^1$, 
\begin{equation*}
j(\Psi_{\alpha}(T)z)=\PP_{\alpha}(T,\phi,r)	
\end{equation*}
For any $M\in \Z$, we define the homotopy $\HM\colon \mathbb{S}^1\times [0,1]\to \R^2$ as
\begin{equation*}
\HM(\phi,\lambda):=j\bigl(\HH(\phi,\lambda)(r^* \cos \phi, -r^* \sin \phi)\bigr)-(\phi+2M\pi,0)
\end{equation*}
For $\lambda=0$ we have
\begin{equation*}
\HM(\phi,0)=\PP_{\A_1}(T,\phi,r^*)-(\phi+2M\pi,0)
\end{equation*}
and $\deg \HM(\cdot,0)$ is known by Corollary \ref{corol:deg_linear} and depends only on $i_0$.
On the other hand, denoting with $\PP(t,\phi,r)$ the Poincaré map in polar coordinates of \eqref{eq:HSsec_order}, for $\lambda=1$ we have
\begin{equation*}
\HM(\phi,1)=\PP(T,\phi,r^*)-(\phi+2M\pi,0)
\end{equation*}

To complete the proof, we have to show that \begin{equation*}
\mydeg\bigl(\PP(T,\cdot,\cdot)-(\phi+2M\pi,0),r^*\bigr)=\mydeg\bigl(\PP_{\A_1}(T,\cdot,\cdot)-(\phi+2M\pi,0),r^*\bigr)\end{equation*}
 Since the homotopy $\HM$ is continuous, this equality follows by Corollary \ref{corol:homotopy}, provided that  that $\HM(\phi,\lambda)\neq 0$ for every $(\phi,\lambda)\in \mathbb{S}^1\times [0,1]$. To show this latter property, let us suppose by contradiction that there exists $(\phi',\lambda')\in \mathbb{S}^1\times [0,1]$ such that $\HM(\phi',\lambda')=0$. Then 
\begin{equation*}
\HH(\phi',\lambda')(r^* \cos \phi', -r^* \sin \phi')=\Pi(\phi'+2M\pi,r^*)=(r^* \cos \phi', -r^* \sin \phi')
\end{equation*}
which is false since $\HH(\phi',\lambda')\in \Sp(1)\setminus \Gamma^0$, which implies that $\HH(\phi',\lambda')z\neq z$ for every $z\in \R^2\setminus\{0\}$.

\end{proof}

\subsection*{Linear-like behaviour at infinity}

 Analogously to the case of linear-like behaviour at the origin, we replace the condition \ref{cond:qinfty} at infinity with the following one.
	\begin{enumerate}[label={\textup{(Q$_\infty^*$)}}]
	\item There exist $\bar r>0$ and two continuous, $T$-periodic functions $b_1(t),b_2(t)$ such that $b_1(t)< q(t,x) < b_2(t)$ for every $\abs{x}>\bar r$, $t\in[0,T]$. Moreover, the matrices 
	\begin{equation} 
	\B_1(t):=\begin{pmatrix}
	b_1(t) & 0\\ 0 & 1
	\end{pmatrix}\qquad 
	\B_2(t):=\begin{pmatrix}
	b_2(t) & 0\\ 0 & 1
	\end{pmatrix}
	\end{equation}
	are $T$-nonresonant with Maslov index $i_\infty=i_{\B_1}=i_{\B_2}$. \label{cond:qinfty_star}
\end{enumerate}

We remark that similar conditions at infinity have been proposed for higher dimensional Hamiltonian systems in \cite{FM1,FM2,Liu}, and are used to estimate rotations in  second order ODEs, e.g.~in~\cite{BG,DZ}. 

Before stating our multiplicity result, let us first state some properties of the dynamics of second order ODEs.

\begin{lemma} \label{lemma:crossing}
	Let $z(t)=(x(t),y(t))$ be a solution of \eqref{eq:HSsec_order}, with $q$ satisfying \ref{cond:q1_star}, \ref{cond:q0_star} and \ref{cond:qinfty_star}. We set
\begin{equation*}
	\Lambda:=\Lambda(\bar r)=\max \left\{\abs{q(t,x)} : (t,x)\in[0,T]\times [-\bar r,\bar r] \right\}
\end{equation*}
Then, for every time $t_1$ such that $x(t_1)\in[-\bar r,\bar r]$ and
\begin{align*}
\bar y:=y(t_1)>\Lambda \bar r T +\frac{2\bar r}{T}
&&\left(\text{resp.}\quad 
\bar y:=y(t_1)<-\Lambda \bar r T -\frac{2\bar r}{T}
\right)
\end{align*}
there exists a time $t_2\in \,]t_1,t_1+T[\,$ such that $x(t_2)=\bar r$ (resp. $x(t_2)=-\bar r$), $x(t)$ is strictly increasing (resp. decreasing) on $[t_1,t_2]$ and
\begin{equation*}
t_2-t_1\leq \frac{2\bar r}{\abs{\bar y}-\Lambda\bar rT}
\end{equation*}
\end{lemma}
\begin{proof}
We prove only one case, since the other one is analogous.	Let us define the set $E\subset \R$ as
\begin{equation*}
E:=\left\{t\in([t_1,t_1+T[\,: x(s)\in[-\bar r,\bar r] \quad \text{for every $s\in[t_1,t]$}\right\}
\end{equation*}
The set $E$ is nonempty and, for every $t\in E$, we have
\begin{equation*}
y(t)=\bar y -\int_{t_1}^{t}q(s,x(s))x(s)\dd s\geq \bar y-\Lambda\bar r T>0
\end{equation*}
Hence
\begin{equation}\label{eq:stime_crossing}
x(t)-x(t_1)=\int_{t_1}^{t}y(s)\dd s\geq (t-t_1)(\bar y-\Lambda\bar r T)>\frac{2(t-t_1)\bar r}{T}	
\end{equation}
From this follows that
$t_2:=\sup E<T$ and $x(t_2)=\bar r$. Moreover, since $y(t)>0$ for $t\in E$, we obtain that $x(t)$ is strictly increasing on $[t_1,t_2]$. Finally, by the first inequality in \eqref{eq:stime_crossing} for $t=t_2$, we have
\begin{equation*}
t_2-t_1\leq \frac{x(t_2)-x(t_1)}{\bar y-\Lambda \bar r T}
\end{equation*}
completing the proof.
\end{proof}

\begin{lemma} \label{lemma:elastic} 
	Let us fix $q\colon \R\times \R\to \R$, satisfying 
	\ref{cond:q1_star}, \ref{cond:q0_star} and \ref{cond:qinfty_star}. We then choose any $\hat{q}\colon \R\times \R\to \R$, satisfying \ref{cond:q1_star}, \ref{cond:qinfty_star} for the same $\bar r$ as $q$, and such that $\hat{q}(t,x)=q(t,x)$ for every $(t,x)\in [0,T]\times[-\bar{r},\bar{r}]$. Then there exists $\hat{r}>0$ such that every $T$-periodic solution $z(t)=(x(t),y(t))$ of
	\begin{equation}\label{eq:HSsec_order_mod}
	\begin{cases}
	\dot x= y\\
	\dot y= -\hat q(t,x)x
	\end{cases}
	\end{equation}
	satisfies $\abs{z(t)}< \hat r$ for every $t\in[0,T]$. Moreover, the constant $\hat r$ depends only on $\bar r, b_1,b_2$ and on the restriction of $q$ to $[0,T]\times[-\bar{r},\bar{r}] $, but not on the choice of $\hat q$.	
\end{lemma}

\begin{proof}
We prove the Lemma by contradiction. Suppose that, for every $\hat r>0$  there exists $\hat q$ as in the statement for which the system \eqref{eq:HSsec_order_mod} admits a $T$-periodic solution whose orbit is not contained in $\BB(0,\hat r)$. By the elastic property for second order ODEs (cf.~\cite{DZ}), we can recover a sequence of functions $\hat q_n$ as in the statement, each admitting a $T$-periodic solution $\hat z_n$ such that $\abs{\hat z_n(t)}>n+C_1$ for every $t\in[0,T]$, where we pick the constant $C_1$ as 
\begin{equation*}
C_1=\left(\Lambda T+\frac{2}{T}+1\right)\bar r \qquad \text{where}\quad \Lambda:=\max \{\abs{q(t,x)}: (t,x)\in[0,T]\times [-\bar r,\bar r]\}
\end{equation*}

Let us now consider the sets $J_n=\{t\in[0,T]: \abs{\hat x_n(t)}\leq \bar r\}$. We claim that $\meas J_n\to 0$. First of all, we show that each set $J_n$ is a union of at most $k$ disjoint intervals, where $k$ is independent of $n$. Let us denote with $t\to (\hat{\phi}_n(t),\hat r_n(t))$ the lift of each solution $\hat x_n(t)$ to the strip. Since all the maps $t\to\hat q(t,\hat z_n(t))$ are uniformly bounded by  $\max \{b_2(t),\Lambda\}$, by the theory of linear second order ODES we obtain an uniform bound on their rotation around the origin, namely there exists an integer $C_2>0$ such that $\abs{\hat\phi_n(0)-\hat\phi_n(T)}<2\pi C_2$ for every $n\in \N$. By Lemma \ref{lemma:crossing}, each connected component of $\Pi^{-1}\bigl([-\bar r,\bar r]\times \R \setminus \BB(0,C_1) \bigr)\subseteq \R\times [C_1,+\infty[\,$ can be crossed only in one direction, hence only once, by each orbit $t\to (\hat{\phi}_n(t),\hat r_n(t))$. Hence each set $J_n$ is the sum of at most $C_2+1$ disjoint intervals, which we write as
\begin{equation*}
J_n=\bigcup_i\, [t^{i,n}_1,t^{i,n}_2]
\end{equation*}
Moreover, Lemma \ref{lemma:crossing} provides also an estimate of the length of each of these intervals, so that, for $n\to+\infty$,
\begin{equation} \label{eq:meas_Jn}
\meas J_n \leq (C_2+1)\frac{2\bar r}{C_1+n-\Lambda\bar rT} \to 0	
\end{equation}

\medbreak
Let us now define $\bar q_n(t):=\hat q_n(t,\hat x_n(t))$ and set
\begin{equation*}
	\tilde  q_n(t):=\begin{cases}
	\bar q_n(t) & \text{ for $t\in [0,T]\setminus J_n$}\\
	\frac{t-t^{i,n}_1}{t^{i,n}_2-t^{i,n}_1}\bar q_n(t^{i,n}_1)+	\frac{t^{i,n}_2-t}{t^{i,n}_2-t^{i,n}_1}\bar q_n(t^{i,n}_2)& \text{ for $t\in[t^{i,n}_1,t^{i,n}_2]$}
	\end{cases}
\end{equation*}
Analogously we set, for $k=1,2$
\begin{equation*}
b_k^n(t):=\begin{cases}
b_k(t) & \text{ for $t\in [0,T]\setminus J_n$}\\
\frac{t-t^{i,n}_1}{t^{i,n}_2-t^{i,n}_1}b_k(t^{i,n}_1)+	\frac{t^{i,n}_2-t}{t^{i,n}_2-t^{i,n}_1}b_k(t^{i,n}_2)& \text{ for $t\in[t^{i,n}_1,t^{i,n}_2]$}
\end{cases}
\end{equation*}
and observe that, by construction
\begin{equation*}
	b_1^n(t)<\tilde  q_n(t)<b_2^n(t) \qquad \text{for every $t\in[0,T]$}
\end{equation*}
Moreover, by \eqref{eq:meas_Jn}, we have, for $1\leq p<\infty$,
\begin{align}\label{eq:approx_conv}
	\norm{\tilde  q_n-\bar q_n}_{L^p([0,T])}\to 0 &&
	\norm{ b_k^n-b_k}_{\CC^0([0,T])}\to 0
\end{align}

Let us fix $\epsilon>0$ such that the planar systems of the form \eqref{eq:HSsec_order} with coefficients $b_1(\cdot)-\epsilon$ and $b_2(\cdot)+\epsilon$ have corresponding $T$-Maslov index $i_\infty$. Let us then define, for $1<p<\infty$, the set $K\subseteq L^p([0,T])\cap L^\infty([0,T])$:
\begin{equation*}
K:=\{q\in L^p([0,T]): b_1(t)-\epsilon\leq q(t)\leq b_2(t)+\epsilon \quad \text{a.e.~on $[0,T]$}\}
\end{equation*}
We observe that the set $K$ is closed, convex, bounded in $L^p$, and therefore it is weakly compact. Furthermore,  by \eqref{eq:approx_conv}, there exists $\bar n\in \N$ such that, for every $n\geq \bar n$, we have $b_1^n,b_2^n\in K$.

Let us now consider the map $\psi_T\colon L^p([0,T])\cap L^\infty([0,T])\to \Sp(1)$, associating to each function $q$ the evolution matrix $\psi(q):=\Psi_q(T)$ of the system  \eqref{eq:HSsec_order} at time $T$. We note that the map $\psi$ is continuous with respect to the  $L^p$-weak topology on its domain, and any matrix norm on $\Sp(1)$. Hence the set $\psi(K)$ is compact and connected in $\Sp(1)$. Moreover, by Proposition \ref{prop:Maslov_monot} we obtain that $\psi(K)$ is contained in one open nonresonant component $\Gamma^\bullet\in\{\Gamma^+,\Gamma^-\}$ of $\Sp(1)$. 

Thus, by \eqref{eq:approx_conv} and the properties of $\psi$, we deduce that there exists $n^*>\bar n$ such that, for every $n\geq n^*$, $\psi(\tilde q_n)\in \Gamma^\bullet$. To see this, let us set $\delta=\dist(\psi(K), \Sp(1)\setminus \Gamma^\bullet)>0$, where the distance is positive due to the compactness of $\psi(K)$. Suppose now by contradiction that $\psi(\bar q_n)$ is not definitively in $\Gamma^\bullet$. Then, using also the weakly compactness of $K$, we can construct a subsequence $n_k$ such that $\norm{\psi(\tilde q_{n_k})-\psi(\bar q_{n_k})}>\delta$ and $\tilde q_{n_k}\rightharpoonup w\in K$. By \eqref{eq:approx_conv} it follows that $\bar q_{n_k}\rightharpoonup w$ and, by the continuity of $\psi$ with respect to the weak $L^p$ topology of the domain, we have $\psi(\bar q_{n_k})\to \psi(w)$ and $\psi(\tilde q_{n_k})\to \psi(w)$, contradicting $\norm{\psi(\tilde q_{n_k})-\psi(\bar q_{n_k})}>\delta$.
 
Therefore $\psi(\bar q_{n})\notin \Gamma_0$ for $n\geq n^*$, meaning that the linear systems of the form \eqref{eq:HSsec_order} with coefficients $\bar q_n(t)$  do not have any nontrivial $T$-periodic solution. This  contradicts the existence of the sequence $\hat z_n$ of $T$-periodic solutions, since each $\hat z_n$ solves also \eqref{eq:HSsec_order} with coefficient $\bar q_n(t)$. The Lemma is thus proved.

\end{proof}

\begin{theorem} \label{th:linlike_inf}
	Let us consider the Hamiltonian system \eqref{eq:HSsec_order}  and assume that \ref{cond:q1_star}, \ref{cond:q0_star}, with $i_\A=i_0$ for the matrix $\A(t)$ of \eqref{eq:def_matrix}, and \ref{cond:qinfty_star} are satisfied. Then the same conclusions of Theorem \ref{th:main} hold.
\end{theorem}

\begin{proof}
	
	Let $\hat r$ be the radius provided by Lemma \ref{lemma:elastic} .
	We set
	$$
	b(t)=\frac{b_1(t)+b_2(t)}{2}
	$$
	and denote with $\B(t)$ the associated matrix for the corresponding planar system, as in \eqref{eq:def_matrix}. 	Clearly, by Corollary \ref{cor:conf_matrix}, we have  $i_\B=i_\infty$.
	
	Let $\mu\colon [0,1]\to [0,1]$ be a $\CC^\infty$ function such that $\mu(0)=0$, $\mu(1)=1$ and $\mu'(0)=\mu'(1)=0$.
	We define $\hat q \colon \R\times \R\to \R$ as
	\begin{equation*}
	\hat{q}(t,z):=\begin{cases}
	q(t,z) & \text{for $\abs{z}\leq\hat r$}\\
	(1-\mu(\abs{z}-\hat r))q(t,z) + \mu(\abs{z}-\hat r)b(t) & \text{for $\hat r<\abs{z}<\hat r+1$}\\
	b(t) & \text{for $\abs{z}\geq\hat r+1$}
	\end{cases}
	\end{equation*}	
	We observe that the function $\hat{q}$ satisfies 	\ref{cond:q1_star},\ref{cond:q0_star} and \ref{cond:qinfty}. Hence, by Theorem \ref{th:linlike}, the same conclusions of Theorem \ref{th:main} hold for  system \eqref{eq:HSsec_order_mod}.
	
	To complete our proof, we have just to show that all the $T$-periodic solutions of \eqref{eq:HSsec_order_mod} are also solutions of \eqref{eq:HSsec_order}. To do so, we observe that $\hat q$ satisfies all the assumptions of Lemma \ref{lemma:elastic}. Hence, every $T$-periodic solutions of \eqref{eq:HSsec_order_mod} is contained in the ball $\BB(0,\hat r)$. Within this ball, systems \eqref{eq:HSsec_order_mod} and \eqref{eq:HSsec_order} coincide, thus all the $T$-periodic solutions of \eqref{eq:HSsec_order_mod} are also solutions of~\eqref{eq:HSsec_order}.
	
\end{proof}

We notice that, although in a different and more sophisticated setting, the general strategy of modifying the Hamiltonian function near infinity, in order to reduce a complex twist behaviour to an asymptotically linear system, has also been recently employed in \cite{FG,FU}, to obtain higher dimensional generalizations of the Poincaré--Birkhoff Theorem.

\begin{remark}[Linear-like behaviour in planar systems] \label{rem:linlike_planar} Although our exposition regards second order ODEs, linear-like behaviour can be similarly studied for general planar Hamiltonian systems of form \eqref{eq:HSplanar}, by requiring the existence of two matrices $\A_1,\A_2\in\CC([0,T],\R^{2\times 2})$, $T$-nonresonant with the same Maslov index $i_0$ (resp.~$i_\infty$), such that
	\begin{equation} \label{eq:planar_llcond}
		\A_1(t)\leq D_zH(t,z) \leq \A_2
	\end{equation}
holds for every $t\in[0,T]$ and every $z$ with $\abs{z}$ sufficiently small (resp. sufficiently large).
Indeed, the proof of Theorem \ref{th:linlike} actually develops in the planar form \eqref{eq:planar_llcond}, and the structure as second order ODE is nowhere used.

The adaptation of the condition at infinity in Theorem \ref{th:linlike_inf}, in the spirit of \cite{Liu}, is instead slightly more delicate, since the prolongation technique shall be applied directly on the Hamiltonian function, in order to preserve the symplectic structure.

\end{remark}

\begin{remark}\label{rem:counter}
The case of linear-like behaviour should also help the comprehension of the counterexample to \cite{MRZ} proposed in  \cite{CMMR_04}. Our approach shows clearly that 
(asymptotic) linearity assures two main qualitative properties of the Poincaré map of the system: one purely rotational, and the other expressed in terms of topological degree. Only combining both these properties we can obtain a full multiplicity result. Indeed, we observe that the modified Poincaré--Birkhoff Theorem in \cite{MRZ} considers not only rotational properties, but also implicitly recovers a degree condition by combining, in a neighbourhood of the origin, the weak twist condition with the area preserving assumption. However,  area preservation cannot be exploited if we consider weak twist on the outer boundary, and indeed we can see that the example of \cite{CMMR_04} fails the degree conditions we employ in this paper. With the notation above, the system in \cite{CMMR_04} is such that $\deg \FF(\cdot,r)=0$ for $r$ sufficiently large, 
whereas with Lemma \ref{lemma:rot_infty} we show that an asymptotically linear system with the same weak twist behaviour necessarily satisfies  $\deg\FF(\cdot,r)=2$.
 
Our results for linear-like behaviour shows that  asymptotic linearity is not necessary; however, in situations of additional weak twist, rotational properties alone are no longer sufficient, and the assumptions on the systems shall also suitably characterize the topological degree of the Poincaré map. 
\end{remark}

\medbreak

\paragraph{Acknowledgments} 
The authors are supported by FCT--Fundação para a Ciência e Tecnologia, under the project UID/MAT/04561/2013. P.G. has been also partially 
supported by the Gruppo Nazionale per l’Analisi Matematica, la Probabilità e le loro Applicazioni
(GNAMPA) of the Istituto Nazionale di Alta Matematica (INdAM).

The authors thank Carlota Rebelo for the useful discussions.

\end{document}